\documentclass[12pt]{amsart}
\usepackage{geometry,blkarray}
\newtheorem{theorem}{Theorem}[section]
\newtheorem*{maintheorem*}{Theorem 1.1}
\newtheorem{cor}[theorem]{Corollary}
\newtheorem{prop}[theorem]{Proposition}
\newtheorem{definition}[theorem]{Definition}
\newtheorem{lemma}[theorem]{Lemma}
\newcommand{\polya}{P\"olya}
\newcommand{\cc}{\overline c}
\newcommand{\rr}{\overline r}

\begin{document}

\bibliographystyle{plain}

\title{Invertibility of Submatrices of the Pascal Matrix and Birkhoff Interpolation}
\author{Scott Kersey}
\address{Georgia Southern University, USA}
\email{scott.kersey@gmail.com}
\keywords{Pascal matrix, Birkhoff interpolation, \polya{} system}
\subjclass{15A15, 41A10}

\maketitle

\begin{abstract}
The infinite upper triangular Pascal matrix is $T = [\binom{j}{i}]$ for $0\leq i,j$.
It is easy to see that any leading principle square submatrix is triangular with 
determinant $1$, hence invertible.
In this paper, we investigate the invertibility of arbitrary square submatrices $T_{r,c}$
comprised of rows $r=[r_0,\ldots,r_m]$ and columns $c=[c_0,\ldots,c_m]$ of $T$.
We show that $T_{r,c}$ 
is invertible iff $r \leq c$ (i.e., $r_i \leq c_i$ for $i=0, \ldots, m$), 
or equivalently, iff all diagonal entries are nonzero.
To prove this result we establish a connection between 
the invertibility of these submatrices and polynomial interpolation.
In particular, we apply the theory of Birkhoff interpolation and \polya{} systems.
\end{abstract}

Paper appears in: Journal of Mathematical Sciences: Advances and Applications, 41, 45--56 (2016).

\section{Introduction}

The infinite and order $n+1$ upper triangular Pascal matrices are
$$
T := \Big [ \binom{j}{i} \Big ] = 
\begin{bmatrix}
1 & 1 & 1 & 1 & \cdots  \\
0 & 1 & 2 & 3 & \cdots  \\
0 & 0 & 1 & 3 & \cdots  \\
0 & 0 & 0 & 1 & \cdots \\
\vdots & \vdots & \vdots & \vdots & \ddots  \\
\end{bmatrix}
\quad\text{and}\quad
T_{n} := 
\begin{bmatrix}
1 & 1 & 1 & \cdots & \binom{n}{0} \\
0 & 1 & 2 & \cdots & \binom{n}{1} \\
0 & 0 & 1 & \cdots & \binom{n}{2} \\
\vdots  & \vdots  & \vdots  & \ddots  & \vdots \\
0 & 0 & 0 & \cdots & \binom{n}{n} 
\end{bmatrix},
$$
for $i,j=0,1, \ldots$, with $\binom{j}{i}:=0$ if $i>j$.
Note that $T_{n}$ is a finite square submatrix of $T$
determined by the first $n+1$ rows and columns.
Other (square) submatrices of $T$ are determined by selecting sequences
of the rows and columns of $T$.
These submatrices can be represented as
$$
T_{r,c} = 
\Big [ \binom{c_j}{r_i} \Big ] = 
\begin{bmatrix}
\binom{c_0}{r_0} & \binom{c_1}{r_0} & \cdots & \binom{c_m}{r_0} \\
\binom{c_0}{r_1} & \binom{c_1}{r_1} & \cdots & \binom{c_m}{r_1} \\
\vdots & \vdots & \ddots & \vdots \\
\binom{c_0}{r_m} & \binom{c_1}{r_m} & \cdots & \binom{c_m}{r_m}
\end{bmatrix}
$$
for some \emph {selections} $r = [r_0, \ldots, r_m]$ and 
$c = [c_0, \ldots, c_m]$ of the rows and 
columns of $T$, respectively.

The main issue considered in this paper is invertibility of these submatrices.
It is trivial to see that $T_n$ is invertible, with determinant $1$.
But what about other submatrices?
For example, 
$$
T_{[0,1,2],[1,2,5]} = \begin{bmatrix} 1&1&1\\1&2&5\\0&1&10 \end{bmatrix}
\quad \text{and} \quad
T_{[1,3,4],[1,2,5]} = \begin{bmatrix} 1&2&5\\0&0&10\\0&0&5 \end{bmatrix}.
$$
The second matrix, a triangular matrix with a zero on the main diagonal, clearly has
a zero determinant, hence is not invertible.
However, while the first matrix is invertible, it is not obvious to see.
In this paper, we will prove the following necessary and sufficient conditions
for invertibity of these submatrices:

\begin{theorem}
\label{t1}
Submatrices $T_{r,c}$ of the upper Pascal matrix are invertible iff 
the following equivalent conditions hold:
\begin{itemize}
\item $r \leq c$ (i.e., $r_i\leq c_i$ for all $i$).
\item There is no zero diagonal entry.
\end{itemize}
\end{theorem}

To our knowledge, this result has not appeared in the literature,
other than in a preliminary draft of this paper in \cite{K13b}.
The motivation for this work comes from our work on dual basis
in subspaces which first appeared in \cite{K13a}.
In that paper we use the theorem to demonstrate the linear independence
of certain dual basis functions.

As it turns out, the problem of invertibility of submatrices of the Pascal matrix
is connected to the theory of Birkhoff interpolation.
This goes back to a paper of G. \polya{} in \cite{P31},
as used by J.M. Whittaker in \cite{W35}, 
and later generalized by D. Ferguson in \cite{F69}.
The more general Birkhoff interpolation problem 
(with more than two interpolation points)
was presented by G. D. Birkhoff in \cite{B1906}.
In our work, we show that such submatrices of the Pascal matrix are invertible iff
a certain 2-point Birkhoff interpolation problem satisfies the \polya{} condition.

More recently, in \cite{P09} and \cite{P11}, Birkhoff
interpolation is generalized to lacunary interpolation,
which could be applied to the problem investigated in this paper.
While some of the techniques in their work are similar to ours,
we were not aware of their work during the preparation of this paper,
and so our results were obtained independently.
Finally, while the work in this paper originates to \cite{K13b},
much of the original material has been removed and some
content has been changed and/or corrected.
In particular, in the original draft we refer to the submatrices
as \emph{truncations} of the Pascal matrix.

The remainder of this paper is organized as follows:
\begin{itemize}
\item
In section (2), we show that our problem is equivalent to
a certain two-point polynomial interpolation problem.
\item
In Section (3), we summarize pertinent aspects of Birkhoff interpolation and
\polya{} systems.
\item
In section (4), we prove Theorem \ref{t1}.
\end{itemize}

\section{The Pascal Matrix and Polynomial Interpolation}

We begin this paper by establishing a connection between
(submatrices of) the Pascal matrix and polynomial interpolation.
Let $\Lambda_\alpha$ be a row vector 
$$
\Lambda_\alpha = \Big[\frac{\delta_\alpha}{0!}, \frac{\delta_\alpha D}{1!}, 
\ldots, \frac{\delta_\alpha D^i}{i!}, \ldots\Big]
$$
comprised of the functionals 
$$
\delta_{\alpha} D^i : f \mapsto f^{(i)}(\alpha),
$$
and let $V$ be the sequence of monomials
$$
V = \Big[1,(\cdot),(\cdot)^2, \ldots\Big].
$$
Note that $\Lambda_\alpha^TV$ is an infinite matrix.
We will be interested in the maps $\Lambda_0$ and $\Lambda_1$, when $\alpha =0$ and $\alpha=1$, respectively.
In particular, we have the following:

\begin{prop}
\label{p1}
The Pascal matrix $T  =  [\binom{j}{i}]$ is equivalent to 
$$\Lambda_1^TV = \Big [ \dfrac{\delta_1D^i (\cdot)^j}{i!}  \Big].$$
\end{prop}

\begin{proof}
Let $A := \Lambda_1^TV$.
We will show that $A = T$.
For $j\geq i$, 
$$
A(i,j) = \frac{\delta_1 D^i}{i!} x^j = 
\frac{1}{i!} \frac{j!}{(j-i)!} x^{j-i} \Big |_{x=1} 
= \binom{j}{i}.
$$
For $j< i$, 
$$
A(i,j) = \frac{\delta_1 }{i!} (D^i x^j) = \frac{\delta_1 }{i!} 0 = 0.
$$
Hence, $A=T$.
\end{proof}


We can state a similar result for submatrices of the Pascal matrix.
Let $r = [r_0, \ldots, r_m]$ and $c = [c_0, \ldots, c_m]$ be selections
of the rows and and columns of $T$, respectively.
Let
$$
\Lambda_{\alpha,r} = \Big[\frac{\delta_\alpha D^{r_0}}{r_0!},  \ldots,
\frac{\delta_\alpha D^{r_m}}{r_m!}  \Big]
$$
and
$$
V_c = \Big[(\cdot)^{c_0}, \ldots, (\cdot)^{c_m}\Big].
$$
Then, by restricting the the result in
Proposition \ref{p1} to the row $r$ and columns $c$, we obtain the following:

\begin{prop}
\label{p2}
Let $r=[r_0, \ldots, r_m]$ and $c=[c_0, \ldots, c_m]$ be selections of the 
rows and columns of $T$.
Then, $$T_{r,c} = \Lambda_{1,r}^TV_c.$$
\end{prop}

Hence, to establish the invertibility of $T_{r,c}$, we can 
analyze $\Lambda_{1,r}^T V_c$.
To do so, we will establish a connection with two-point polynomial interpolation,
which we do by augmenting the point functionals at $x=1$ with functionals 
at $x=0$.
Let $\cc$ be the complement of $c$ in $[0:n]$, 
with $n := \max\{r_m,c_m\}$.
Let
$$
\Lambda_{r,\cc} := [\Lambda_{1,r}, \Lambda_{0,\cc}],
$$
and 
$$
V_{c,\cc} = [V_c,V_{\cc}]
= [(\cdot)^{c_0}, \ldots, (\cdot)^{c_d}, 
   (\cdot)^{\cc_0}, \ldots, (\cdot)^{\cc_{n-m}}].
$$
Note that $V_{c,\cc}$ is the power basis of degree $n$ with the terms rearranged.

\begin{theorem}
\label{t2}
$T_{r,c}$ is invertible iff $\Lambda_{r,\cc}^T V_{c,\cc}$ is invertible.
\end{theorem}

\begin{proof}
The interpolation matrix can be expressed in block form as follows:
$$
\Lambda_{r,\cc}^T V_{c,\cc}
=
[\Lambda_{1,r}, \Lambda_{0,\cc}]^T [V_c,V_{\cc}]
= \left[ \begin{array}{c|c} 
\Lambda_{1,r}^T V_c & \Lambda_{1,r}^T V_{\cc} \\
\hline
\Lambda_{0,\cc}^T V_c & \Lambda_{0,\cc}^T V_{\cc} 
\end{array} \right].
$$
By Proposition \ref{p2}, the upper left block is $T_{r,c}$.
The lower left and lower right blocks comprises terms of the form
$
\frac{1}{i!} \delta_0 D^j (\cdot)^i  = \frac{1}{i!} D^j (\cdot)^i \Big|_0,
$
which are $1$ when $i=j$ and zero otherwise.
But since $c$ and $\cc$ are complementary,
the lower left block is all zeros, while the lower right block is the identity.
Hence,
$$
\Lambda_{r,\cc}^T V_{c,\cc} =
\left[ \begin{array}{c|c} T_{r,c} & \Lambda_{1,r}^T V_{\cc} \\ 
  \hline 0 & I \\ \end{array} \right].
$$
Therefore, 
$\det(\Lambda_{r,\cc}^T V_{c,\cc}) = \det(T_{r,c})$,
and so 
$\Lambda_{r,\cc}^T V_{c,\cc}$
is invertible iff $T_{r,c}$ is invertible.
\end{proof}

For example, suppose $r = [0,2,4,7]$ and $c = [1,2,5,8]$ in $[0:8]$.
Then, 
\begin{align*}
\cc &= [0,3,4,6,7], \\
\Lambda_{r,\cc} &= 
\begin{bmatrix} \frac{\delta_1}{0!} & 
   \frac{\delta_1 D^2}{2!}& \frac{\delta_1 D^4}{4!}& \frac{\delta_1 D^7}{7!} &
  \frac{\delta_0}{0!}& \frac{\delta_0D^3}{3!} & 
  \frac{\delta_0D^4}{4!}& \frac{\delta_0D^6}{6!}& \frac{\delta_0D^7}{7!}\end{bmatrix}, \\
V_{[c,\cc]} &=  \begin{bmatrix} (\cdot)&(\cdot)^2&(\cdot)^4 &
   (\cdot)^8&1&(\cdot)^3&(\cdot)^4&(\cdot)^6&(\cdot)^7 \end{bmatrix}.
\end{align*}
And so, 
\[
\Lambda_{r,\cc}^T V_{c,\cc} = 
\begin{blockarray}{cccccccccc}
 &t&t^2&t^5&t^8&1&t^3&t^4&t^6&t^7\\
\begin{block}{c[cccc|ccccc]}
\frac{1}{0!}\delta_1&\binom10&\binom20& \binom50& \binom{8}{0} 
   & \binom{0}{0}& \binom{3}{0}& \binom{4}{0}& \binom{6}{0}& \binom{7}{0}\\
\frac{1}{2!}\delta_1D^2&\binom12&\binom22& \binom52& \binom{8}{2} 
   & \binom{0}{2}& \binom{3}{2}& \binom{4}{2}& \binom{6}{2}& \binom{7}{2}\\
\frac{1}{4!}\delta_1D^4&\binom14&\binom24& \binom54& \binom{8}{4} 
   & \binom{0}{4}& \binom{3}{4}& \binom{4}{4}& \binom{6}{4}& \binom{7}{4}\\
\frac{1}{7!}\delta_1D^7&\binom17&\binom27& \binom57& \binom{8}{7} 
   & \binom{0}{7}& \binom{3}{7}& \binom{4}{7}& \binom{6}{7}& \binom{7}{7}\\
\cline{2-10}
\frac{1}{0!}\delta_0 & 0&0&0&0&1&0&0&0&0 \\
\frac{1}{3!}\delta_0D^3 & 0&0&0&0&0&1&0&0&0 \\
\frac{1}{4!}\delta_0D^4 & 0&0&0&0&0&0&1&0&0 \\
\frac{1}{6!}\delta_0D^6 & 0&0&0&0&0&0&0&1&0 \\
\frac{1}{7!}\delta_0D^7 & 0&0&0&0&0&0&0&0&1\\ 
\end{block}
\end{blockarray}\].

\section{Birkhoff Interpolation and the \polya{} condition}

In the previous section, we established an equivalent condition
for the invertibility of the submatrix $T_{r,\cc}$ of the Pascal matrix.
This equivalent condition was expressed in terms of a kind
of generalized Vandermonde determined by the evaluation of certain
derivatives at $0$ and $1$.
As it turns out, this kind of interpolation problem is called 
(2-point) Birkhoff interpolation.
To solve problems like this, Ferguson (\cite{F69}) used \emph{incidence matrices}.

\begin{definition}
\label{d1}
An \emph{incidence matrix} $E$ for 2-point interpolation problems on 
$\Pi_n$ is a $2 \times (n+1)$ matrix
$$
E := \Big [ e_{ij} \Big ]
= \begin{bmatrix}
e_{00} & e_{01} & \cdots & e_{0m} \\
e_{10} & e_{11} & \cdots & e_{1m}
\end{bmatrix}
$$
of ones and zeros, with exactly $n+1$ ones.
The term $e_{ij}$ is $1$ when the interpolation problem includes $\delta_{x_i} D^j$.
\end{definition}

For example, the following are incidence matrices, each of dimension $2 \times 6$ with
exactly $6$ ones:
$$
\begin{bmatrix} 0 & 1 & 0 & 1 & 0 & 0\\ 1 & 0 & 1 & 1 & 0 & 1 \end{bmatrix},
\quad
\begin{bmatrix} 0 & 1 & 0 & 0 & 1 & 1\\ 0 & 1 & 1 & 0 & 1 & 0 \end{bmatrix}.
$$
These correspond to the functionals
$$[\delta_{x_0}, \delta_{x_0}^3, \delta_{x_1}, 
   \delta_{x_1}^2, \delta_{x_1}^3, \delta_{x_1}^5]$$
and
$$[\delta_{x_0}, \delta_{x_0}^4, \delta_{x_0}^5, 
   \delta_{x_1}, \delta_{x_1}^2, \delta_{x_1}^4],$$
respectively.
In this paper, we take $x_0=0$ and $x_1=1$.
As in \cite{F69}, we define $M_j$ to be the cumulative column sum
\begin{equation}
\label{e1}
M_j := \sum_{i=0}^1 \sum_{k=0}^j e_{ik}
\end{equation}
for $j=0:n$. 
In the previous examples, $M = [1,2,3,5,5,6]$ and $M = [0,2,3,3,5,6]$.
Note that $M_5=6$ is the total number of ones in both examples.
In general, $M_n = n+1$ for any incidence matrix.

The problem of 2-point Birkhoff interpolation was studied by \polya.
With respect to the incidence matrices, the following defintion is used.

\begin{definition} [\cite{F69}]
\label{d2}
The incidence matrix $E$ satisfies the \emph{\polya{} condition} if $M_j > j$ for $j=0:m$.
\end{definition}

In the above two examples, the first matrix is \polya, however the second is not because
$M_0=0< 1$ and $M_3=3<4$.
The following result, proved independently by \polya{} and Whittaker (as also described in \cite{F69}), 
gives necessary and sufficient conditions for correct interpolation.

\begin{theorem}[adapted from \cite{P31}, \cite{W35}]
\label{t3}
Let $E$ be a $2 \times (n+1)$ incidence matrix with entries $e_{ij}$, ones or zeros.
Let
$$\Lambda := [ \delta_{x_i}D^j : e_{ij} = 1]$$
and $ V^n = [1, (\cdot), \ldots, (\cdot)^n]$.  
Then, the system $\Lambda^T V^n$ is invertible iff $E$ satisfies the \polya{} condition.
\end{theorem}

\section{Invertibility of Submatrices of the Pascal Matrix}

Suppose $r = [r_0, \ldots, r_m]$ and $c = [c_0, \ldots, c_m]$ are selections
(increasing sequences) of the rows and columns of $T$, respectively.
Let $\cc$ be the complement of $c$ in $[0:n]$ with $n := \max\{r_m,c_m\}$.
We define 
$E_{r,\cc}$ to be the $2 \times (n+1)$ matrix with
$E_{r,\cc}(0,j) = 1$ if $r_i = 1$,
and
$E_{r,\cc}(1,j) = 1$ if $\cc_i = 1$,
with all other entries $0$.
For example, let $r = [0,1,4]$ and $c = [0,4,5]$.
We choose $n=5$.
Then, $\cc = [1,2,3]$ and
$$
E_{r,\cc} =  E_{[0,1,4],[1,2,3]} 
  = \begin{bmatrix} 1&1&0&0&1&0\\0&1&1&1&0&0 \end{bmatrix}.
$$
The following verifies that $E_{r,\cc}$ is an incidence matrix.

\begin{lemma}
\label{l1}
$E_{r,\cc}$ is an incidence matrix of dimension $2 \times (n+1)$.
\end{lemma}

\begin{proof}
Since,
$$\#\cc = (n+1) - \#c = (n+1) - (m+1) = n-m,$$
it follows that
$$
\sum_{i=0}^1 \sum_{j=0}^n e_{ij} = \#r + \#\cc = (m+1)+(n-m) = n+1.
$$
Hence, there are exactly $n+1$ ones (and $n+1$ zeros).
Following Definition \ref{d1}, $E_{r,\cc}$ is an incidence matrix.
\end{proof}

The next results are used to prove Theorem \ref{t1}.
The first is a Corollary of two previous theorems.

\begin{cor}
\label{c2}
The Pascal submatrix $T_{r,c}$ is invertible iff $E_{r,\cc}$ is \polya.
\end{cor}

\begin{proof}
By Theorem \ref{t2}, 
$T_{r,c}$ is invertible iff $\Lambda_{r,\cc}^T V_{c,\cc}$ is invertible,
and by Theorem \ref{t3}, $\Lambda_{r,\cc}^T V_{c,\cc}$ is invertible
iff $E_{r,\cc}$ is \polya.
Therefore, $T_{r,c}$ is invertible iff $E_{r,\cc}$ is \polya.
\end{proof}

\begin{lemma} 
\label{l2}
$r \leq c$ iff $\cc \leq \rr$.
\end{lemma}

\begin{proof}
Suppose $r \leq c$.
By the tautologies
\begin{align*}
\rr-\cc &= \rr\cap c = c\cap \overline{r\cap c} \\
\cc-\rr &= \cc\cap r = r\cap \overline{r\cap c},
\end{align*}
we see that
$\rr$ contains those elements in $c$ that are not in both $c$ and $r$, 
and $\cc$ contains those elements in $r$ that are not in both $c$ and $r$.
Since $r \leq c$, it follows that $\rr \geq \cc$.

For the reverse implication, assume that $\cc \leq \rr$,
and apply a similar proof as above, noting that
$\overline{\overline{r}} = r$ and $\overline{\overline{c}} = c$.
\end{proof}

\begin{lemma}
\label{l3}
Let $r$ and $x$ be selection vectors of the same length.
Then, $E_{r,x}$ is \polya{} iff $x \leq \rr$.
\end{lemma}

\begin{proof}
Suppose first that $x=\rr$.
We want to establish that $E_{r,\rr}$ is \polya.
Since $r$ and $\rr$ are complementary,
each column of the incidence matrix $E_{r,\rr}$ has exactly one ``1" and one ``0".
Hence, $M_0 = 1$, $M_1 = 1+1=2$, etc.
In particular $M_j = j+1$ for $j=0, \ldots, n$,
implying that $E_{r,\rr}$ is \polya.

Now, suppose that $x \leq \rr$.
Let $M^x := [M_0^x, \ldots, M_n^x]$ and $M^{\rr} := [M_0^{\rr}, \ldots, M_n^{\rr}]$ 
be the cumulative sums defined in (\ref{e1}).
Since $x \leq \rr$, it follows that 
$$M^x_j \geq M^{\rr}_j = j+1.$$
Hence, $E_{r,x}$ is \polya.

Now suppose that $x \not\leq \rr$.
Then, $x_i > \rr_i$ for some $i$.
Suppose that $j$ is the first such occurrence.
Then, $M_{j-1}^x = j = M_j^x$.
Hence, we don't have $M_j^x > j$ for all $j$, 
implying that $E_{r,x}$ is not \polya.
By contraposition, $E_{r,x}$ \polya{} implies $x \leq \rr$.

On combining these last two results,
we have $E_{r,x}$ \polya{} iff $x \leq \rr$.
\end{proof}

We now establish our main result.

\begin{maintheorem*}
Submatrices $T_{r,c}$ of the upper Pascal matrix are invertible iff 
the following equivalent conditions hold:
\begin{itemize}
\item $r \leq c$ (i.e., $r_i\leq c_i$ for all $i$).
\item There is no zero diagonal entry.
\end{itemize}
\end{maintheorem*}

\begin{proof}
For the first part, 
$r \leq c$ iff $\cc \leq \rr$ by Lemma \ref{l2},
and $\cc \leq \rr$ iff $E_{r,\cc}$ is \polya{} by Lemma \ref{l3}. 
Hence, $E_{r,\cc}$ is  \polya{} iff $r \leq c$.
By Corollary \ref{c2}, this holds iff $T_{r,c}$ is invertible.
This establishes the first part.

For the second part,
note that the diagonal elements of $T_{r,c}$ 
are of the form $\binom{c_k}{r_k}$,
which are zero iff $r_k > c_k$.
Hence, a diagonal element is zero iff $r \not\leq c$.
By the first result, this occurs iff $T_{r,c}$ is not invertible.
By contraposition,
$T_{r,c}$ is invertible iff there is no zero on the diagonal.
\end{proof}


\end{document}